\documentclass[12pt,a4paper]{article}
\usepackage{mathrsfs}
\usepackage{epsfig, graphicx,color}
\usepackage{latexsym,amsfonts,amsbsy,amssymb}
\usepackage{amsmath,amsthm}
\usepackage{hyperref,color}
\makeatletter 

\@addtoreset{equation}{section}
\makeatother 
\allowdisplaybreaks 

\newtheorem{theorem}{Theorem}[section]

\newtheorem{remark}{Remark}[section]

\begin{document}
\title{A Small Note About Lower Bound of Eigenvalues\footnote{This work is supported in part
by the National Natural Science Foundation of China (NSFC 91330202, 11371026,
11001259, 11031006, 2011CB309703),  the National 
Center for Mathematics and Interdisciplinary Science, CAS and the
President Foundation of AMSS-CAS.}}
\author{Hehu Xie\footnote{LSEC, NCMIS, Institute
of Computational Mathematics, Academy of Mathematics and Systems
Science, Chinese Academy of Sciences, Beijing 100190,
China(hhxie@lsec.cc.ac.cn)}\ \ \ and \ \
Chunguang You\footnote{LSEC, ICMSEC, Academy of Mathematics and Systems Science, Chinese Academy of
Sciences, Beijing 100190, China (youchg@lsec.cc.ac.cn)}
}
\date{}
\maketitle
\begin{abstract}
This paper gives a way to produce the lower bound of eigenvalues defined in a Hilbert space 
by the eigenvalues defined in another Hilbert space. The method is based on using the 
max-min principle for the eigenvalue problems.


\end{abstract}
\section{Introduction}
Recently, there are more and more research about the lower bound of eigenvalues of some 
type of partial differential operators (see all the reference in this paper). 
In this paper, we give a framework to produce 
the lower bound of the eigenvalues defined in some spaces.  We would like to say that 
the results in this not is obtained in our seminar and the motivation is to give a type 
of framework for the results in \cite{Liu}. 

\section{Abstract framework}
The assumptions for function spaces are listed below
\begin{itemize}
\item[(A1)]
Let $\mathbb{X}$ and $\mathbb{Y}$ be two Hilbert spaces  with inner product and norm
$(\cdot,\cdot)_\mathbb{X}$, $\|\cdot\|_\mathbb{X}$ and
$(\cdot,\cdot)_\mathbb{Y}$, $\|\cdot\|_\mathbb{Y}$ respectively.
And their exits a continuous and compact embedding operator $\gamma :\mathbb{X}\mapsto\mathbb{Y}$.

\item[(A2)]
Bilinear form $M(u,v)$ is symmetric, continuous and coercive over the space $\mathbb{X}\times\mathbb{X}$;
bilinear form $N(p,q)$ is symmetric, continuous and semi-positive definite
over the space $\mathbb{Y}\times\mathbb{Y}$.
\end{itemize}

\begin{remark}
$M(\cdot,\cdot)$ is an inner product of $\mathbb{X}$ with corresponding norm
$\|\cdot\|_M:=\sqrt{M(\cdot,\cdot)}$, and
$N(\cdot,\cdot)$ is an inner product of $\mathbb{Y}\backslash\ker (N)$ with corresponding norm
$\|\cdot\|_N:=\sqrt{N(\cdot,\cdot)}$.
\end{remark}

In the rest of this paper, for any $x\in\mathbb{X}$, we just use $x$ rather than $\gamma x$ to 
present the corresponding element in $\mathbb{Y}$.

Consider the abstract eigenvalue problem: Find $(\lambda,u)\in \mathbb{R}\times \mathbb{X}$,
such that $N(u,u) = 1$ and
\begin{equation}\label{abstract problem}
M(u,v)=\lambda N(u,v), \,\,\forall\, v\in \mathbb{X}.
\end{equation}
From the compactness (see, e.g. Section 8 of Babuska {Babuska-Osborn-1991}),
(\ref{abstract problem}) has the eigenpairs $\{(\lambda_k,u_k)\}$ ($k=1,2,\cdots$)
with $0<\lambda_1\leq\lambda_2\leq\cdots$ and $N(u_i, u_j) = \delta_{ij}$ ($\delta_{ij}:$ Kronecker's delta).

Let $\mathbb{W}$ and $\mathbb{V}$ be two subspaces of $\mathbb{X}$. Then we could define the eigenvalue problems on
$\mathbb{W}$ and $\mathbb{V}$, respectively.

Find $(\lambda^\mathbb{W},u^\mathbb{W})\in \mathbb{R}\times \mathbb{W}$,
such that $N(u^\mathbb{W},u^\mathbb{W}) = 1$ and
\begin{equation}\label{eigen W}
M(u^\mathbb{W},v^\mathbb{W})=\lambda^\mathbb{W} N(u^\mathbb{W},v^\mathbb{W}), \,\,\forall\, v^\mathbb{W}\in \mathbb{W}.
\end{equation}
Let $\{(\lambda_k^\mathbb{W},u_k^\mathbb{W})\}$ ($k=1,2,\cdots$) be the eigenpairs of (\ref{eigen W})
with $0<\lambda_1^\mathbb{W}\leq\lambda_2^\mathbb{W}\leq\cdots$ and $N(u_i^\mathbb{W}, u_j^\mathbb{W}) = \delta_{ij}$.

Find $(\lambda^\mathbb{V},u^\mathbb{V})\in \mathbb{R}\times \mathbb{V}$,
such that $N(u^\mathbb{V},u^\mathbb{V}) = 1$ and
\begin{equation}\label{eigen V}
M(u^\mathbb{V},v^\mathbb{V})=\lambda^\mathbb{V} N(u^\mathbb{V},v^\mathbb{V}), \,\,\forall\, v^\mathbb{V}\in \mathbb{V}.
\end{equation}
Let $\{(\lambda_k^\mathbb{V},u_k^\mathbb{V})\}$ ($k=1,2,\cdots$) be the eigenpairs of (\ref{eigen V})
with $0<\lambda_1^\mathbb{V}\leq\lambda_2^\mathbb{V}\leq\cdots$ and $N(u_i^\mathbb{V}, u_j^\mathbb{V}) = \delta_{ij}$.

Define
$\ker_{\mathbb{X}}(N):=
\{x\in\mathbb{X}\,|\,
N(x,x)=0\}$,
$\ker_{\mathbb{W}}(N):=
\{w\in\mathbb{W}\,|\,N(w,w)=0\}$ and
$\ker_{\mathbb{V}}(N):=
\{v\in\mathbb{V}\,|\,N(v,v)=0\}$.
Denote $R(\cdot)$ by the Rayleigh quotient over $\mathbb{X}$: for any $x \in\mathbb{X}\backslash\ker_{\mathbb{X}}(N)$,
\begin{equation}
R(x):=\frac{M(x,x)}{N(x,x)}.
\end{equation}
Thus the stationary values and stationary points of $R(\cdot)$ over $\mathbb{W}$ and $\mathbb{V}$ correspond to the eigenpairs of eigenvalue problem (\ref{eigen W}) and (\ref{eigen V}), respectively.
And the min-max principle holds for both $\lambda_k^\mathbb{W}$ and $\lambda_{k}^\mathbb{V}$:
\begin{equation}
\lambda_k^\mathbb{W}=\min_{S_k^\mathbb{W}\subset \mathbb W}\max_{w\in S_k^\mathbb{W}}R(w),\,\,\ \ \ \ \ 
\lambda_k^\mathbb{V}=\min_{S_k^\mathbb{V}}\max_{v\in S_k^\mathbb{V}}R(v),
\end{equation}
where $S_k^\mathbb{W}$ and $S_k^\mathbb{V}$ are any $k$-dimensional subspaces of $\mathbb{W}\backslash\ker_{\mathbb{W}}(N)$ and $\mathbb{V}\backslash\ker_{\mathbb{V}}(N)$, respectively.

Let $P: \mathbb{X}\mapsto\mathbb{V}$ be the projection operator with respect to $M(\cdot,\cdot)$:
\begin{equation}
M(x-P x, v) = 0, \,\,\forall\, x\in\mathbb{X},\,\forall\, v\in\mathbb{V}.
\end{equation}
Then we have the following  theorem which is the main result in this note. 

\begin{theorem}\label{framework theorem}
Suppose there exist a constant number $\alpha$ such that the following inequality holds
\begin{equation}\label{framework projection estimation}
\|x-P x\|_N \leq \alpha \|x-P x\|_M,\,\,\forall\, x\in \mathbb{X}.
\end{equation}
Let $\lambda_k^\mathbb{W}$ and $\lambda_k^\mathbb{V}$ be the ones defined in (\ref{eigen W})
and (\ref{eigen V}). Then, we have
\begin{equation}\label{framework lower bounds}
\frac{\lambda_k^\mathbb{V}}{1+\alpha^2\lambda_k^\mathbb{V}}\leq\lambda_k^\mathbb{W} \,\,\ \ (k=1,2,\cdots).
\end{equation}
\end{theorem}
\begin{proof}
From the argument of compactness mentioned above, both of the min-max and the max-min principle also hold for
$\lambda_k$:
\begin{equation}\label{min-max max-min}
\lambda_k=\min_{S_k}\max_{u\in S_k}R(u)=\max_{S,\dim{(S)}\leq k-1}\min_{u\in S^{\mathbb{X}\perp}}R(u),\,\,k=1,2,\cdots,
\end{equation}
where $S_k$ denotes any $k$-dimensional subspace of $\mathbb{X}\backslash\ker_{\mathbb{X}}(N)$, and
$S^{\mathbb{X}\perp}$ denotes the orthogonal complement of $S$ in $\mathbb{X}$ with respect to $M(\cdot,\cdot)$.

Due to the min-max principle, it is clear that
$\lambda_k^\mathbb{W}\geq\lambda_k$ as $\mathbb{W}\subset\mathbb{X}$.
Choosing a special $k-1$ dimensional subspace
$\mathbb{V}_{k-1}:=\text{span}\{u_1^\mathbb{V},u_2^\mathbb{V},\cdots,u_{k-1}^\mathbb{V}\}$, we could give a lower bound for
$\lambda_k$ from the max-min principle in (\ref{min-max max-min}) by
\begin{equation}\label{max-min lower}
\lambda_k^\mathbb{W}\geq\lambda_k\geq\min_{v\in \mathbb{V}_{k-1}^{\mathbb{X}\perp}}R(v).
\end{equation}
Let $\mathbb{V}_{k-1}^{\mathbb{V}\perp}$ denotes the orthogonal complement of $\mathbb{V}_{k-1}$ in $\mathbb{V}$
with respect to $M(\cdot,\cdot)$, i.e., $\mathbb{V}=\mathbb{V}_{k-1}\oplus\mathbb{V}_{k-1}^{\mathbb{V}\perp}$.
As a consequence, $\mathbb{X}$ can be decomposed by
\begin{equation}
\mathbb{X}
=\mathbb{V}\oplus\mathbb{V}^{\mathbb{X}\perp}
=\mathbb{V}_{k-1}\oplus\mathbb{V}_{k-1}^{\mathbb{V}\perp}\oplus\mathbb{V}^{\mathbb{X}\perp}.
\end{equation}
Then we have
$\mathbb{V}_{k-1}^{\mathbb{X}\perp}=\mathbb{V}_{k-1}^{\mathbb{V}\perp}\oplus\mathbb{V}^{\mathbb{X}\perp}$.

Notice that $\mathbb{V}_{k-1}^{\mathbb{X}\perp}\subset\mathbb{X}$. For any $v\in\mathbb{V}_{k-1}^{\mathbb{X}\perp}$,
we have
\begin{equation}
v=Pv+(I-P)v,\,\,\text{where}\,\,Pv\in\mathbb{V}_{k-1}^{\mathbb{V}\perp},\,\,(I-P)v\in\mathbb{V}^{\mathbb{X}\perp}.
\end{equation}
Then $\displaystyle{\|Pv\|_N^2\leq\frac{\|Pv\|_M^2}{\lambda_k^\mathbb{V}}}$ is held by
\begin{equation}
\lambda_k^\mathbb{V}=\min_{v\in \mathbb{V}_{k-1}^{\mathbb{V}\perp}}R(v)
=\min_{v\in \mathbb{V}_{k-1}^{\mathbb{V}\perp}}\frac{\|v\|_M^2}{\|v\|_N^2}
\leq\frac{\|Pv\|_M^2}{\|Pv\|_N^2}.
\end{equation}
Therefore, we have for any $v\in\mathbb{V}_{k-1}^{\mathbb{X}\perp}$,
\begin{equation}\label{R lower}
\begin{aligned}
R(v)&=\frac{\|v\|_M^2}{\|v\|_N^2}=\frac{\|v\|_M^2}{\|Pv+(I-P)v\|_N^2}
\geq\frac{\|v\|_M^2}{\|Pv\|_N^2+\|v-Pv\|_N^2}\\
&\geq\frac{\|v\|_M^2}{\displaystyle{\frac{1}{\lambda_k^\mathbb{V}}}\|Pv\|_M^2+\alpha^2\|v-Pv\|_M^2}
\geq\frac{\|v\|_M^2}{(\displaystyle{\frac{1}{\lambda_k^\mathbb{V}}}+\alpha^2)(\|Pv\|_M^2+\|v-Pv\|_M^2)}\\
&=\frac{\lambda_k^\mathbb{V}\|v\|_M^2}{(1+\alpha^2\lambda_k^\mathbb{V})(\|Pv\|_M^2+\|v-Pv\|_M^2)}
=\frac{\lambda_k^\mathbb{V}}{1+\alpha^2\lambda_k^\mathbb{V}}.
\end{aligned}
\end{equation}

The conclusion in (\ref{framework lower bounds}) is immediately obtained using (\ref{max-min lower}) and (\ref{R lower}).
\end{proof}

\section{Some Applications}

Based on Theorem \ref{framework theorem}, it is easy to give the lower-bound results for the eigenvalues 
which are computed by both the conforming and nonconforming  finite element methods if the constant 
$\alpha$ in (\ref{framework projection estimation}). 

Here, we suppose $\Omega$ be a domain in $\mathbb R^d$ and let $V_h^{\rm NC}$ denote some type of nonconforming finite element space such that $V_h^{NC}\not\subset V:=H_0^1(\Omega)$. If we take $\mathbb X:= H_0^1(\Omega)+ V_h^{\rm NC}$, $\mathbb V=V_h^{\rm NC}$, $\mathbb W=V$, 
$$M(u,v)=\int_{\Omega}\nabla u\nabla vd\Omega\ \ \ {\rm and}\ \ \  N(u,v)=\int_{\Omega}uvd\Omega.$$
the inequality (\ref{framework lower bounds}) is the result obtained in \cite{CarstensenGedicke,Liu}. For this setting, we 
can choose CR \cite{BrennerScott} and ECR \cite{LinXieLuoLiYang} elements to build the nonconforming finite element space $V_h^{\rm NC}$.  
We would like to say the similar derivatives can be extended to the Biharmonic,  
Stokes and Steklov eigenvalue problems \cite{CarstensenGallistl,HuHuangLin,LinXie_lowerbound,LinXie,Liu}.

When we choose $\mathbb V$ as some type of conforming finite element method such that $\mathbb V\subset \mathbb W$, 
we can also obtain the lower-bound result (\ref{framework lower bounds}) if we can have the upper bound of the 
constant $\alpha$ in (\ref{framework projection estimation}).

\end{document}